\documentclass[10pt]{article}

\usepackage[english]{babel}
\usepackage{amsmath}
\usepackage{amssymb}
\usepackage{bm}
\usepackage{bbm}
\usepackage{mathrsfs,color}
\usepackage{graphics,graphicx,theorem}
\usepackage{dsfont}
\usepackage{setspace}

\usepackage{mathtools}

\usepackage{natbib}
\usepackage{multirow}
\usepackage{hhline}
\usepackage{hyperref}

\usepackage{authblk}

\hypersetup{
	colorlinks=true,
	urlcolor=blue,
	citecolor=blue,
	linktoc=all,
	linkcolor=red} 

\usepackage[scriptsize]{subfigure}
\allowdisplaybreaks

\makeatletter
\renewcommand\@biblabel[1]{}
\makeatother

\def\bSig\mathbf{\Sigma}

\newcommand{\E}{\mathbb{E}}

\renewcommand{\P}{\mathbb{P}}

\newcommand{\R}{\mathds{R}}

\newcommand{\Pcr}{\mathscr{P}}

\newcommand{\ddr}{\mathrm{d}}

\newcommand{\D}{{\rm d}}

\newcommand{\ind}[1]{1\!\!1_{#1}}

\newtheorem{theorem}{Theorem}[section]
\newtheorem{definition}{Definition}[section]
\newtheorem{proposition}{Proposition}[section]
\newtheorem{corollary}{Corollary}[section]

\newenvironment{proof}{\noindent \textit{Proof.}}{\hfill$\square$}

\usepackage[top=1.2in,bottom=1.2in,left=1.2in,right=1.2in]{geometry}

\begin{document}

\title{\bf {\Large{Consistent estimation of the missing mass for feature models}}}

\author[,1]{Fadhel Ayed \thanks{fadhel.ayed@gmail.com}}
\author[,2]{Marco Battiston \thanks{marco.battiston@stats.ox.ac.uk}}
\author[,3]{Federico Camerlenghi \thanks{federico.camerlenghi@unimib.it}}
\author[,4]{Stefano Favaro \thanks{stefano.favaro@unito.it}}

\affil[1]{University of Oxford}
\affil[2]{Lancaster University}
\affil[3]{University of Milano--Bicocca}
\affil[4]{University of Torino}

\date{}
\maketitle
\thispagestyle{empty}

\setcounter{page}{1}

\begin{abstract}
Feature models are popular in machine learning and they have been recently used to solve many unsupervised learning problems. In these models every observation is endowed with a finite set of features, usually selected from an infinite collection $(F_{j})_{j\geq 1}$. Every observation can display feature $F_{j}$ with an unknown probability $p_{j}$. A statistical problem inherent to these models is how to estimate, given an initial sample, the conditional expected number of hitherto unseen features that will be displayed in a future observation. This problem is usually referred to as the missing mass problem. In this work we prove  that, using a suitable multiplicative loss function and without imposing any assumptions on the parameters $p_{j}$, there does not exist any universally consistent estimator for the missing mass. In the second part of the paper, we focus on a special class of heavy-tailed probabilities $(p_{j})_{j\geq 1}$, which are common in many real applications, and we show that, within this restricted class of probabilities, the nonparametric estimator of the missing mass suggested by \cite{Fad(17)} is strongly consistent. As a byproduct result, we will derive concentration inequalities  for the missing mass and the number of features observed with a specified frequency in a sample of size $n$.
\end{abstract}


\maketitle

\section{Introduction}

Feature models generalize species sampling models by allowing every observation to belong to more than one species, now called features. In particular, every observation is endowed with a finite set of features selected from a (possibly infinite) collection of features $(F_{j})_{j\geq1}$. Every feature $F_{j}$ is associated with an unknown probability $p_{j}$, and each observation displays feature $F_{j}$ with probability $p_{j}$. We may conveniently represent each observation with a binary sequence, whose entries indicate the presence (1) or absence (0) of each feature.  Feature models have been first applied in ecology for modeling incidence vectors collecting the presence or absence of species traps (\citet{Col(12)} and \citet{Cha(14)}), and more recently in several fields of biosciences, such as the study of genetic variation and protein interactions (\citet{Chu(06)}, \citet{Ion(09)}, \citet{Ion(10)} and \citet{Zou(16)}). They also found applications in the analysis of choice behaviour arising from psychology, marketing and computer science (\citet{Gor(06)}); in the context of binary matrix factorization for modeling dyadic data to design recommender system (\citet{Mee(07)}); in graphical models (\citet{Woo(06)} and \citet{Woo(07)}); in cognitive psychology for the analysis of similarity judgement matrices (\citet{Nav(07)}); in the context of independent component analysis and sparse factor analysis (\citet{Kno(07)}); in link prediction using network data (\citet{Mil(10)}).

The Bernoulli product model is arguably the most popular feature model. It assumes that the $i$--th observation is a sequence $Y_{i}=(Y_{i,j})_{j\geq1}$ of independent Bernoulli random variables with unknown success probabilities $(p_{j})_{j\geq1}$, and that $Y_{r}$ is independent of $Y_{s}$ for any $r\neq s$. Therefore $X_{n,j}:=\sum_{1\leq i\leq n}Y_{i,j}$, namely the number of times that feature $F_{j}$ has been observed in a sample $(Y_{1},\ldots,Y_{n})$, is a Binomial random variable with parameter $(n,p_{j})$ for any $j\geq1$. Recently, the Bernoulli product model has been extensively applied to the fundamental problem of discovering genetic variation in human populations. See, e.g., \citet{Ion(09)}, \citet{Zou(16)}) and references therein. In such a context, interest is in estimating the conditional expected number, given a sample $(Y_{1},\ldots,Y_{n})$, of hitherto unseen features that would be observed if an additional sample $Y_{n+1}$ was collected, namely
\begin{equation} \label{missing_mass}
M_{n}(Y_{1},\ldots,Y_{n};(p_{j})_{j\geq1})=\E\left[\sum_{j\geq1}\ind{\{ X_{n, j} =0,Y_{n+1,j}=1\}}\,|\,Y_{1},\ldots,Y_{n}\right]=\sum_{j\geq 1} p_j \ind{\{ X_{n, j} =0\}}, 
\end{equation}
where $\ind{}$ is the indicator function. The statistic $M_{n}(Y_{1},\ldots,Y_{n};(p_{j})_{j\geq1})$ is referred to as the missing mass, i.e. the sum of the probability masses of unobserved features in a sample of size $n$. In genetics, interest in estimating \eqref{missing_mass} is motivated by the ambitious prospect of growing databases to encompass hundreds of thousands of genomes, which makes important  to quantify the power of large sequencing projects to discover new genetic variants (\citet{Aut(15)}). An accurate estimate of the missing mass provides a quantitative evaluation of the potential and limitations of these datasets, providing a roadmap for large-scale sequencing projects.

Let $\hat{T}_{n}(Y_{1},\ldots,Y_{n})$ denote an arbitrary estimator of $M_{n}(Y_{1},\ldots,Y_{n};(p_{j})_{j\geq1})$.
For easiness of notation, in the rest of the paper we will not highlight the dependence on $(Y_{1},\ldots,Y_{n})$ and $(p_{j})_{j\geq1}$, and we simply write $M_{n}$ and $\hat{T}_{n}$.  Motivated by the recent works of \citet{Oha(12)}, \citet{Mos(15)}, \citet{Ben(17)} and \citet{Fad(18)} on the estimation of the missing mass in species sampling models, in this paper we consider the problem of consistent estimation of $M_{n}$ under the Bernoulli product model. The classical notion of \textit{additive} consistency, involving the large $n$ limiting behaviour of $\hat{T}_n-M_n$, is not suitable in the context of the estimation of $M_{n}$. This is because $M_n \to 0$, as $n \to +\infty$, which implies that $0$ is a consistent estimator of the missing mass for any sequence $(p_j)_{j \geq 1}$. Hence, in such a framework, one should invoke a more adequate notion of consistency, which allows to achieve more informative results. This notion of consistency is based on the limiting behaviour of the multiplicative loss function
\begin{equation} \label{loss}
L(\hat{T}_{n},M_{n}):=\left| \frac{\hat{T}_{n}}{M_{n}}-1 \right|.
\end{equation}
More precisely  we say that the estimator $\hat{T}_{n}$ is \textit{multiplicative} consistent for $M_{n}$ if $\hat{T}_n/M_n \to 1$ as $n\rightarrow+\infty$, either almost surely or in probability. The multiplicative loss function has been already used in statistics, e.g. for the estimation of small value probabilities using importance sampling (\citet{Cha(18)}) and for the estimation of tail probabilities in extreme value theory (\citet{Bei(99)}).
We show that there do not exist universally consistent estimators, in the multiplicative sense, of the missing mass $M_{n}$. That is, under the Bernoulli product model and the loss function \eqref{loss}, we prove that for any estimator $\hat{T}_{n}$ of $M_{n}$ there exists at least a choice of $(p_{j})_{j\geq1}$ for which $\hat{T}_n/M_n$ does not converge to $1$ in probability, as $n\rightarrow+\infty$. The proof relies on non-trivial extensions of Bayesian nonparametric ideas and techniques developed by \citet{Fad(18)} for the estimation of the missing mass in species sampling models. In particular, the key argument makes use of a generalized Indian Buffet construction (\cite{Jam(17)}), which allows to prove inconsistency by exploiting properties of the posterior distribution of $M_{n}$. Our inconsistency result is the natural counterpart for feature models of the work of \citet{Mos(15)}, showing the impossibility of estimating the missing mass without imposing any structural (distributional) assumption on the  $p_{j}$'s. We complete our study by investigating the consistency of an estimator of $M_{n}$ recently proposed by \citet{Fad(17)}. To the best of our knowledge this is the first nonparametric estimator of $M_{n}$, in the sense that its derivation does not rely on any distributional assumption on the $p_{j}$'s. We show that the estimator of \citet{Fad(17)} is strongly consistent, in the multiplicative sense, under the assumption that the tail of $(p_{j})_{j\geq1}$ decays to zero as a regularly varying function (\citet{Bin87}). The proof relies on novel concentration inequalities for $M_{n}$, as well as for related statistics, which are of independent interest.

The paper is structured as follows. In Section \ref{sec:inconsistency} we prove that for the Bernoulli product model there do not exist universally consistent estimators, in the multiplicative sense, of the missing mass $M_{n}$. Section \ref{sec:consentration} introduces some exponential tail bounds for $M_{n}$, as well as for related statistics, which are then applied in Section \ref{sec:regular_variation} to show that the estimator of $M_{n}$ in \citet{Fad(18)} is consistent under the assumption of regularly varying probabilities $p_{j}$'s.


\section{Non existence of universally consistent estimators of the missing mass} \label{sec:inconsistency}

Consider the Bernoulli product model described in the Introduction. Without loss of generality, we assume that each feature $F_{j}$ is labeled by a value in $[0,1]$ and therefore $(F_{j})_{j\geq 1}$ is a sequence of distinct points in $[0,1]$. Furthermore, the probabilities $(p_{j})_{j\geq 1}$ are assumed to be summable, i.e. $\sum_{j \geq 1}p_j < +\infty$; this condition is needed in order to guarantee that every observation $Y_{i}$ will display only a finite number of features almost surely. Indeed, $\sum_{j \geq 1}p_j < +\infty$ is equivalent to $\sum_{j \geq 1}\P(F_{j} \in Y_{i}) = \sum_{j \geq 1}\E[\ind{\{F_{j} \in Y_{i}\}}] < +\infty$, which in turns implies 
$\sum_{j \geq 1}\ind{\{F_{j} \in Y_{i}\}} < +\infty$ almost surely, by Tonelli-Fubini Theorem. The two unknown sequences $(F_{j})_{j\geq 1}$ and $(p_{j})_{j\geq 1}$ can be uniquely encoded in a finite measure on $[0,1]$, $\sum_{j\geq 1}p_{j} \delta_{F_{j}}(\cdot)$, with all masses smaller than one. We can therefore consider as parameter space the set 
\begin{equation}
\Pcr:= \left\{ \sum_{j\geq 1}p_{j} \delta_{F_{j}}: \; F_{j},p_j \in [0,1], \, \forall j\geq 1,  \, \sum_{j \geq 1} p_j <+\infty \right\}.
\end{equation}
Recall that $X_{n,j}$ denotes the number of times that feature $F_{j}$ has been observed in the sample $(Y_{1},\ldots,Y_{n})$, that is $X_{n,j}= \sum_{1\leq i\leq n}Y_{i,j}=\sum_{1\leq i \leq n} \ind{\{F_{j} \in Y_{i}\}}$ is a Binomial random variable with parameter $(n,p_j)$. For a fixed $n \geq 1$, an estimator $\hat{T}_{n}:[0,1]^{n}\rightarrow \R_{+}$ of the missing mass $M_n$ is a measurable map which argument is the observed sample $\mathbf{Y}_{n}=(Y_1, \ldots , Y_n)$. We say that the estimator $\hat{T}_n$ is multiplicative consistent under the parameter space  $\Pcr$ if for every $\epsilon>0$ and every $p\in \Pcr$,
\begin{equation}\label{eq:consistency}
 \lim_{n \to +\infty } \P_{\mathbf{Y}_{n}|p} \left( \left| \frac{\hat{T}_n}{M_n}-1 \right| \geq \varepsilon  \right) =0,
\end{equation} 
where $\P_{\mathbf{Y}_{n}|p}$ denotes the law of the observations $\mathbf{Y}_{n}$ under a feature allocation model of parameter $p$. Theorem \ref{thm:inconsistency} shows that there are no universally multiplicative consistent estimators of $M_{n}$ for the class $\Pcr$. This means that for any estimator $\hat{T}_n$ of the missing mass, there exists at least one element  $p \in \Pcr$ for which $\hat{T}_n/M_n$ does not converge to $1$ in probability, as $n\rightarrow+\infty$.

\begin{theorem} \label{thm:inconsistency}
Under the feature allocation model, there are no universally consistent estimators, i.e. there are no estimators   satisfying \eqref{eq:consistency}. In particular, for every estimator  $\hat{T}_n$, it is possible to find an element $p \in \Pcr$ such that for any $\varepsilon \in (0, 1/6)$ 
\begin{equation}
\label{eq:inconsistency}
 \limsup_{n \to +\infty } \P_{\mathbf{Y}_{n}|p} \left( \left| \frac{\hat{T}_n}{M_n}-1 \right| \geq \varepsilon  \right) >C.
\end{equation} 
for some strictly positive constant $C$.
\end{theorem}

\subsection{Proof of Theorem \ref{thm:inconsistency}}
In order to prove Theorem \ref{thm:inconsistency}, it is enough to show that for every estimator $\hat{T}_n$ and every $\epsilon \in (0, 1/6)$,
\begin{equation} \label{eq:proof1}
\sup_{p\in \Pcr} \limsup_{n \to +\infty } \P_{\mathbf{Y}_{n}|p} \left( \left| \frac{\hat{T}_n}{M_n}-1 \right| \geq \epsilon  \right) >C,
\end{equation} 
and therefore there exists a $p\in \Pcr$ for which $\hat{T}_n$ is not consistent. 

First, let us notice that, for every $\epsilon \in (0, 1/6)$,
\begin{equation}\label{ineq_inverse}
\sup_{p\in \Pcr} \limsup_{n \to +\infty } \P_{\mathbf{Y}_{n}|p} \left( \left| \frac{\hat{T}_n}{M_n}-1 \right| \geq \epsilon  \right) \geq \sup_{p\in \Pcr} \limsup_{n \to +\infty } \P_{\mathbf{Y}_{n}|p} \left( \left| \frac{M_n}{\hat{T}_n}-1 \right| \geq 2\epsilon  \right).
\end{equation}
Indeed, if $\left|\frac{\hat{T}_n}{M_n} - 1 \right| < \epsilon$, then 
\begin{equation}
-M_n\epsilon  < \hat{T}_n-M_n < M_n\epsilon, \label{ineq:inv_d1}
\end{equation}
and, from the lower bound of \eqref{ineq:inv_d1}, $\hat{T}_n < (1-\epsilon) M_n$. Because $\epsilon < 1/2$, it follows that
$ \frac{1}{\hat{T}_n} < \frac{2}{M_n}$.  This last inequality together with \eqref{ineq:inv_d1}  leads to 
$\left| \frac{M_n}{\hat{T}_n} - 1 \right| < 2\epsilon $. Considering the complements of the two events, it follows that
$$ \left| \frac{M_n}{\hat{T}_n} - 1 \right| \geq 2\epsilon \Rightarrow  \left| \frac{\hat{T}_n}{M_n} - 1 \right| \geq \epsilon, $$
and, as a consequence, $ \mathbb{P}(|\frac{\hat{T}_n}{M_n} - 1| \geq \epsilon) \geq \mathbb{P}(|\frac{M_n}{\hat{T}_n} - 1| \geq 2\epsilon)$, proving (\ref{ineq_inverse}). From now on, we will denote $\varepsilon = 2 \epsilon \in (0,1/3)$ and prove that 
\begin{equation}
\sup_{p\in \Pcr} \limsup_{n \to +\infty } \P_{\mathbf{Y}_{n}|p} \left( \left| \frac{M_n}{\hat{T}_n}-1 \right| \geq 2\epsilon  \right)>C,
\end{equation}
for some strictly positive constant $C$.

The main idea of the proof is in the following formula and works as follows: we lower bound the supremum over $\Pcr$ in \eqref{ineq_inverse} by an average with respect to a (carefully chosen) prior for $p$; we swap the conditional distribution of $\mathbf{Y}_{n}|p$ and the marginal of $p$ with the conditional of $p|\mathbf{Y}_{n}$ and the marginal of $\mathbf{Y}_{n}$; we lower bound the event probability with respect to the posterior of $p$ given $\mathbf{Y}_{n}$.
Formally,
\begin{align} 
\nonumber  & \sup_{p\in \Pcr} \limsup_{n \to +\infty } \P_{\mathbf{Y}_{n}|p} \left( \left| \frac{M_n}{\hat{T}_n}-1 \right| \geq \varepsilon  \right)\geq  
\E_{p} \left[ \limsup_{n \to +\infty}  \P_{\mathbf{Y}_{n}| p}\left( \left| \frac{M_n}{\hat{T}_n}-1 \right| \geq \varepsilon  \right) \right] \\ \label{eq:suplim1}
&\qquad\qquad \geq \limsup_{n \to +\infty } \E_{\mathbf{Y}_{n}} \left[\P_{p| \mathbf{Y}_{n}} \left( \left| \frac{M_n}{\hat{T}_n}-1 \right| \geq \varepsilon  \right) \right]. 
\end{align} 
where we have applied reverse Fatou's lemma to take the $\limsup$ outside the expectation. In \eqref{eq:suplim1}, $\E_{p}$ denotes the expectation with respect to the prior for $p$, $\E_{\mathbf{Y}_{n}}$ the expectation with respect to the marginal distribution of $\mathbf{Y}_{n}$ and $\P_{p| \mathbf{Y}_{n}}$ the probability under the posterior of $p$ given $\mathbf{Y}_{n}$. 

Our choice of the nonparametric prior for $p$ is based on completely random measures (see \cite{Dal(08)}) and the generalized Indian Buffet process prior of \cite{Jam(17)}. In particular, a prior for $p \in \Pcr$ can defined through a completely random measure $\tilde{N}(\cdot)=\sum_{j}s_{j} \delta_{F_j}(\cdot)$ on $[0,1]$, where $(\{s_{j},F_{j}\})_{j\geq 1}$ is  a Poisson Point Process on $\R^+\times[0,1]$, by setting $p(\cdot)=\sum_{j}(1-e^{-s_j}) \delta_{F_j}(\cdot) \in \Pcr$. We select $\tilde{N}$ to be a completely random measure with L\'evy intensity $\nu (\D s, \D F) = e^{-s}/s \,\D s \ind{(0,1)}(F) \D F$. The distribution of 
$\tilde{N}$ is completely characterized by its Laplace functional defined as follows,
\begin{equation*}
\E \left[ e^{-\int_{ [0,1]} f( F) \tilde{N}( \D F)} \right]  = \exp \left\{- \int_{\R^+\times [0,1]} (1-e^{-sf(F)}) \nu (\D s, \D F) \right\},
\end{equation*}
for any measurable function $f:  [0,1] \to \R^+$. See also \cite{Kin(93)}.

Theorem 3.1 of \cite{Jam(17)} provides with a distributional equality for the posterior of $\tilde{N}$ given $\mathbf{Y}_{n}$. Denoting by $F_1^*, \ldots , F_{k_n}^*$ the $k_n$ distinct features observed in $\mathbf{Y}_{n}$, we have the following distributional equality 

\begin{equation}
\label{eq:posterior}
\tilde{N}| \mathbf{Y}_{n} \stackrel{d}{=} \tilde{N}_n+\sum_{\ell=1}^{k_n}  J_\ell \delta_{F_\ell^*}
\end{equation}
where the $J_\ell$'s are non-negative random jumps and $\tilde{N}_n$ is an independent  completely random measure with updated L\'evy intensity $\nu_n (\D s , \D F) = e^{-s n} \nu (\D s , \D F)$. \\

Defining $A_n := \{F_1^*, \ldots , F_{k_n}^*  \} $,  from  \eqref{eq:posterior} we have that, for any Borel set $B$ in $\R^+$, the missing mass $M_n$ satisfies
\begin{equation}
\begin{split}
\P_{p|\mathbf{Y}_{n}} (M_n \in B)&= \P_{p|\mathbf{Y}_{n}} \left( \sum_{j \geq 1} p_j \delta_{F_{j}} (A_n^c) \in B  \right)\\
& = \P_{\tilde{N}|\mathbf{Y}_{n}} \left( \sum_{j \geq 1} (1-e^{-s_{j}}) \delta_{F_{j}} (A_n^c) \in B   \right)\\
  & = \P_{\tilde{N}|\mathbf{Y}_{n}} \left( \int_{  A_n^c }(1-e^{-s}) \tilde{N}( \D F) \in B  \right)\\
& \stackrel{\eqref{eq:posterior}}{=}  
\P_{\tilde{N}_n} \left( \int_{ [0,1] }(1-e^{-s}) \tilde{N}_n( \D F) \in B  \right)
\end{split}
\end{equation}
showing that the posterior  distribution of the missing mass $M_n$ is equal in distribution to the random variable $\int_{ [0,1] }(1-e^{-s}) \tilde{N}_n( \D F) $. Besides, it is worth to introduce the random variable 
\[
S_n: = \tilde{N}_{n}([0,1])= \int_{[0,1]} s \tilde{N}_{n} ( \D F)
\] 
whose distribution can be computed exactly and turns out to be a Gamma random variable of parameters $(1,n+1)$. Indeed, from the Laplace functional,  for every $x\in \R$ we have
\begin{align*}
\E [e^{xS_n}] &= \E \left[ \exp \left\{ x\int_{[0,1]} s \tilde{N}_{n} ( \D F) \right\} \right]\\
& = \exp \left\{- \int_0^1 \int_0^{+\infty} (1-e^{xs}) e^{-sn} \nu (\D s, \D \theta) \right\}\\
& = \exp \left\{ - \int_0^{+\infty} (1-e^{xs}) \frac{e^{-s(n+1)}}{s} \D s \right\}
= \left( 1-\frac{x}{n+1} \right)^{-1},
\end{align*} 
which is the characteristic function of a Gamma$(1,n+1)$ random variable.\\

We now have all the necessary ingredients to prove the lower bound \eqref{eq:suplim1}. Fix $\varepsilon \in (0,1/3)$. First note that, the inverse triangular inequality entails
\begin{equation} \label{eq:triangular}
\begin{split}
\left| \frac{M_n}{\hat{T}_n} -1 \right| &= \left| \frac{M_n}{S_n} \left(\frac{S_n}{\hat{T}_n} -1+1 \right) -1 \right|
\geq \left| \frac{M_n}{S_n} \left|\frac{S_n}{\hat{T}_n} -1  \right|-\left|\frac{M_n}{S_n}-1  \right| \right|\\
& \geq \frac{M_n}{S_n} \left|\frac{S_n}{\hat{T}_n} -1  \right|-\left|\frac{M_n}{S_n}-1  \right|
\end{split}
\end{equation}
which implies
\begin{equation} \label{eq:prob_ineq1}
\P_{p| \mathbf{Y}_{n}} \left(  1-\frac{\varepsilon}{2}\leq \frac{M_n}{S_n}\leq 1 \, , \,
\Big| \frac{S_n}{\hat{T}_n}-1\Big| > 3\varepsilon \right)  \leq \P_{p| \mathbf{Y}_{n}} \left( \Big| \frac{M_n}{\hat{T}_n}-1 \Big| > \varepsilon\right) 
\end{equation}
indeed, thanks to \eqref{eq:triangular}, the two events together 
\[
 1-\frac{\varepsilon}{2}\leq \frac{M_n}{S_n}\leq 1 \, , \,
\Big| \frac{S_n}{\hat{T}_n}-1\Big| >3 \varepsilon
\] imply that
\[
\left| \frac{M_n}{\hat{T}_n} -1 \right| \geq \frac{M_n}{S_n} \left|\frac{S_n}{\hat{T}_n} -1  \right|-\left|\frac{M_n}{S_n}-1  \right| \geq \left(1-\frac{\varepsilon}{2}\right)3\varepsilon -\frac{\varepsilon}{2} =\varepsilon(5-3\varepsilon)/2
> \varepsilon\]
where the last inequality follows from the fact that $\varepsilon <1$. Hence, from \eqref{eq:prob_ineq1}, we have that
\begin{equation*}
 \P_{p| \mathbf{Y}_{n}} \left( \left| \frac{M_n}{\hat{T}_n}-1 \right| > \varepsilon \right)  \geq 
 \P_{p| \mathbf{Y}_{n}} \left(  1-\frac{\varepsilon}{2}\leq \frac{M_n}{S_n}\leq 1\right)-1+ \P_{p| \mathbf{Y}_{n}} \left(
\left| \frac{S_n}{\hat{T}_n}-1\right| > 3\varepsilon \right) 
\end{equation*}
which may be plugged into \eqref{eq:suplim1} to obtain
\begin{equation}\label{eq:twoterms}
\begin{split}
& \sup_{p\in \Pcr} \limsup_{n \to +\infty } \P_{\mathbf{Y}_{n}|p} \left( \left| \frac{\hat{T}_n}{M_n}-1 \right| \geq \varepsilon  \right)\\
& \qquad \qquad \geq 
 \limsup_{n \to +\infty}  \E_{\mathbf{Y}_{n}} \left[\P_{p| \mathbf{Y}_{n}} \left(  1-\frac{\varepsilon}{2}\leq \frac{M_n}{S_n}\leq 1 \right) -1\right] \\
 & \qquad\qquad\qquad\qquad\qquad\qquad\qquad +\inf_{ x >0} \inf_{\mathbf{Y}_{n}}\P_{p| \mathbf{Y}_{n}}
 \left(
\left| \frac{S_n}{x}-1\right| > 3\varepsilon \right)  .
 \end{split}
\end{equation}
We are going to lower bound separately the two terms on the r.h.s. of \eqref{eq:twoterms}. With regard to the first term, let us observe that the elementary inequality
$x-x^2/2\leq 1-e^{-x}\leq x$, for $x>0$, implies that for all $j\geq 1$
\begin{equation*} \label{eq:exp_bound}
s_j -\frac{1}{2}s_j^2 \leq 1-e^{-s_j} \leq s_j.
\end{equation*}
Summing over $j$,
\begin{equation*}
S_n-\frac{1}{2}S_n^2 \leq S_n -\frac{1}{2}\sum_{j \geq 1} s_j^2 \delta_{F_{j}}(A_n^c) \leq M_n \leq S_n,
\end{equation*}
and therefore,
\begin{equation*}
1-\frac{1}{2}S_n\leq \frac{M_n}{S_n}\leq 1.
\end{equation*}
As a simple consequence of the last inequality, for any $\varepsilon>0$, the event $\{ S_n\leq \varepsilon \}$ implies the validity of 
$\{1-\varepsilon/2\leq M_n/S_n\leq 1\}$ and therefore we can upper bound the first term in \eqref{eq:twoterms} as follows
\begin{equation} \label{eq:first_term}
\begin{split}
&\P_{p| \mathbf{Y}_{n}} \left(  1-\frac{\varepsilon}{2}\leq \frac{M_n}{S_n}\leq 1 \right) -1  \geq \P_{p| \mathbf{Y}_{n}} \left( S_n \leq\varepsilon \right) -1 \\ 
& \qquad\qquad\qquad =  (n+1)\int_0^\varepsilon e^{-x(n+1)} \D x -1 = -e^{-\varepsilon (n+1)},
\end{split}
\end{equation}
where we have used the fact that the posterior distribution of $S_{n}$ is ${\rm Gamma} (1, n+1)$. \\

Let us now consider the second term on the r.h.s. of \eqref{eq:twoterms}. Using again the fact that $S_n$ is Gamma distributed and $\varepsilon<1/3$, we have
\begin{align*}
\P_{p| \mathbf{Y}_{n}} 
 \left(\left| \frac{S_n}{x}-1\right| > 3\varepsilon \right) & = 1- (n+1)\int_{(1-3\varepsilon)x}^{(1+3\varepsilon)x} e^{-s(n+1)}\D s \\
 & = 1+e^{-(1+3\varepsilon)x(n+1)}-e^{-(1-3\varepsilon)x(n+1)}\\
& \geq \inf_{y >0}\left[ 1+e^{-(1+3\varepsilon)y}-e^{-(1-3\varepsilon)y} \right]
\end{align*}
it is now easy to see that the function $f(y) := 1+e^{-(1+3\varepsilon)y}-e^{-(1-3\varepsilon)y} $ is strictly positive, continuous and  admits a global minimum on $\R^+$ at the point $y^*=\log((1+3\varepsilon)/(1-3\varepsilon))/(6\varepsilon)$, therefore 
\begin{equation}
\label{eq:second_term}
\P_{p| \mathbf{Y}_{n}} 
 \left(\left| \frac{S_n}{x}-1\right| > 3\varepsilon \right)\geq f(y^*)=:C>0.
\end{equation}
Using the two bounds \eqref{eq:first_term} and \eqref{eq:second_term} in \eqref{eq:twoterms}, for any $\varepsilon \in (0,1/3)$ we get
\begin{equation*}
\sup_{(p_j)_{j\geq 1}\in \Pcr} \lim_{n \to +\infty} \P\left(\left|\frac{M_n}{\hat{T}_n}-1\right| \right)
 \geq -\limsup_{n \to +\infty} e^{-\varepsilon (n+1)} + C =C >0
\end{equation*}
which completes the proof.




\section{Concentration inequalities for feature models} \label{sec:consentration}

In this section we will establish exponential tail bounds for the missing mass $M_n$ 
and the statistic $K_{n,r}$ defined by
\[
K_{n,r} = \sum_{j \geq 1} \ind{\{X_{n,j} = r\}}, \quad \text{for } r \geq 1
\]
which counts the number of features observed with frequency $r$  in the sample
$\mathbf{Y}_{n}$. The statistic $K_{n,r}$ is of interest in different applications of feature allocation models and its analysis will be important for the study of the estimator of missing mass considered in Section \ref{sec:regular_variation}, which involves $K_{n,1}$.
The tail bounds we present in this Section are valid in full generality, i.e. without any assumptions on the probability masses $(p_j)_{j \geq 1}$. 
In Section \ref{sec:regular_variation}, we will use these results to prove consistency results under the assumption of regularly varying heavy tails $(p_j)_{j \geq 1}$.\\
In order to derive the concentration inequalities for $K_{n,r}$ we will use Chernoff bounds, which require suitable bounds on the log-Laplace transform. First, let us recall some definitions from \cite{Bou(13)} and \cite{Ben(17)}.
\begin{definition}
Let $X$ be  a real--valued random variable defined on some probability space, then:
\begin{itemize}
\item[i.] $X$ is \textit{sub-Gaussian} on the right tail (resp. on the left tail) with variance factor $v$ if for any $\lambda\geq 0$ (resp. $\lambda\leq 0$)
\begin{equation}
\label{eq:def_sub-gaussian}
\log \E \left( e^{\lambda (X-\E [X])} \right)\leq \frac{v\lambda^2}{2};
\end{equation}
\item[ii.] $X$ is \textit{sub-Gamma} on the right tail with variance factor $v$ and scale parameter $c$ if
\begin{equation}
\label{eq:def_sub-gamma}
\log \E \left[ e^{\lambda (X-\E [X])} \right]\leq \frac{\lambda^2 v}{2 (1-c\lambda)}, \quad \text{for any }\lambda \text{ satisfying } 0 \leq \lambda \leq 1/c;
\end{equation}
\item[iii.] $X$ is \textit{sub-Gamma} on the left tail with variance factor $v$ and scale parameter $c$
if $-X$ is sub-gamma on the right tail with variance factor $v$ and scale parameter $c$;
\item[iv.] $X$ is \textit{sub-Poisson} with variance factor $v$ if for all $\lambda \in \R$
\begin{equation}
\log \E \left[ e^{\lambda (X-\E [X])} \right]\leq \phi (\lambda) v
\end{equation}
being $\phi (\lambda) = e^{\lambda}-1-\lambda$.
\end{itemize}
\end{definition}
Note that a sub-Gaussian random variable is also sub-Gamma for any choice of the scale parameter $c$, but in general the inverse is not true. As we will see in the sequel, the bounds on the log-Laplace \eqref{eq:def_sub-gaussian}--\eqref{eq:def_sub-gamma} imply exponential tails bounds by means of the Chernoff inequality.
See \cite{Bou(13)} for the details.\\
The following proposition shows that the missing mass $M_n$ is sub-Gaussian on the left tail and sub-Gamma on the right one.
\begin{proposition} \label{prop:laplace}
Let $n >2$. On the left tail, the random variable $M_n$ is sub-Gaussian with variance factor $v_n^{-} := 2 \E [K_{n+2,2}]/((n+2)\cdot (n+1))$, i.e. for any $\lambda \leq 0$ it holds 
\begin{equation}
\label{eq:sub-gaussian}
\log \E \left[ e^{\lambda [M_n-\E [M_n]]} \right] \leq \frac{\lambda^2v_n^-}{2}.
\end{equation}
On the right tail, the random variable $M_n$ is sub-Gamma with variance factor $v_n^+ := 2\E [K_n] /(n^2-2n)$ and scale parameter $1/n$, i.e. for any $0\leq \lambda <1/n$ one has
\begin{equation}
\label{eq:sub-gamma}
\log \E \left[ e^{\lambda [M_n-\E [M_n]]} \right] \leq \frac{\lambda^2 v_n^+}{2(1-\lambda/n)}.
\end{equation}
\end{proposition}
\begin{proof}
We first focus on the proof of \eqref{eq:sub-gaussian}. Let $\lambda \leq 0$, exploiting the independence of the random variables $X_{n,j}$'s and the elementary inequality $\log(z)\leq z-1$, valid for any $z >0$, we obtain
\begin{align*}
\log \E \left[e^{\lambda [M_n-\E [M_n]]} \right] & = \sum_{j \geq 1} \log \E \left[ e^{\lambda (p_j \ind{\{X_{n,j}=0\}} -p_j \P (X_{n,j}=0))} \right]\\
& = \sum_{j \geq 1} \left( -\lambda p_j \P (X_{n,j}=0) +\log (e^{\lambda p_j}\P (X_{n,j}=0)+1-\P (X_{n,j}=0)) \right)\\
&\leq  \sum_{j\geq 1} \P (X_{n,j}=0) (e^{\lambda p_j}-1-\lambda p_j).
\end{align*}
We  observe that, being $\lambda \leq 0$, one has:
\begin{align*}
\log \E \left[ e^{\lambda [M_n-\E [M_n]]} \right] &\leq \sum_{j \geq 1} \P (X_{n,j}=0) \frac{\lambda^2p_j^2}{2}
= \frac{\lambda^2}{2} \sum_{j \geq 1} p_j^2 (1-p_j)^n\\
& = \frac{\lambda^2}{2} \frac{2}{(n+1)(n+2)} \E [K_{n+2,2}]= \frac{\lambda^2v_n^-}{2}
\end{align*} 
hence \eqref{eq:sub-gaussian} has been proven.\\
We now concentrate on the proof of \eqref{eq:sub-gamma}, arguing exactly as before we obtain that
\begin{align*}
\log \E \left[ e^{\lambda [M_n-\E [M_n]]} \right] & \leq  \sum_{j\geq 1} \P (X_{n,j}=0) (e^{\lambda p_j}-1-\lambda p_j) = \sum_{k \geq 2} \sum_{j \geq 1} \frac{(\lambda p_j)^k}{k!} (1-p_j)^n\\
& = \sum_{k \geq 2} \frac{\lambda^k}{k!} \sum_{j \geq 1} p_j^k e^{-np_j} = 
\sum_{k \geq 2} \left(\frac{\lambda}{n} \right)^k \sum_{j \geq 1} \frac{(np_j)^k}{k!} e^{-np_j}
\end{align*}
where we have used the infinite series representation for the exponential function. 
Fixing  the useful notation 
\[
\Phi_{n,k} := \sum_{j\geq 1} \frac{(np_j)^k}{k!} e^{-np_j}, \qquad \Phi_n:=\sum_{j \geq 1} (1-e^{-np_j})
\] 
and observing that $\Phi_{n,k} \leq \Phi_n$, for any $k \geq 1$, we get 
\begin{equation}
\log \E \left[ e^{\lambda [M_n-\E [M_n]]} \right]  \leq 
\sum_{k \geq 2} \left(\frac{\lambda}{n} \right)^k \Phi_{n,k} \leq
\Phi_n \sum_{k \geq 2} \left(\frac{\lambda}{n} \right)^k = \Phi_n \frac{\lambda^2}{n^2(1-\lambda/n)} \label{eq:bound_phin}
\end{equation}
for any $0 < \lambda <1/n$. Proceeding along similar lines as in \cite[Lemma 1]{Gne(07)}, it is not difficult to see that
\[
|\Phi_n-\E [K_n]| \leq \frac{2}{n}  \Phi_{n,2}\leq \frac{2}{n}\Phi_n,
\]
which entails $\Phi_n\leq \E [K_n]/(1-2/n)$, for any $n>2$. The last inequality can be used to provide an upper bound for the r.h.s. of \eqref{eq:bound_phin} as follows
\begin{equation*}
\log \E \left[ e^{\lambda [M_n-\E [M_n]]} \right]  \leq 
\frac{\E [K_n]}{(1-2/n)}\frac{\lambda^2}{n^2(1-\lambda/n)} = \frac{\lambda^2 v_n^+}{2 (1-\lambda/n)}
\end{equation*}
and \eqref{eq:sub-gamma} has been now proved.
\end{proof}
As already mentioned at the beginning of this section, the sub-Gaussian and sub-Gamma bounds obtained in Proposition \ref{prop:laplace} imply useful exponential tail bounds for $M_n$
(see \cite{Bou(13)}). More specifically we have that:
\begin{corollary} \label{cor:concentration_M}
For any $n >2$ and $x \geq 0$, the following hold
\begin{align*}
\P (M_n-\E [M_n] \leq -x) &\leq \exp \left\{ -\frac{x^2}{2v_n^-}  \right\}, \\
\P (M_n-\E [M_n] \geq  x) & \leq \exp \left\{ -v_n^+n^2 \left[ 1+\frac{x}{nv_n^+} -\sqrt{1+\frac{x}{nv_n^+}} \right] \right\}.
\end{align*}
\end{corollary}
\begin{proof}
The two inequalities follow by the Chernoff bound and the log-Laplace bound proved in Proposition \ref{prop:laplace}. This is a standard argument, see \cite{Bou(13)} for details.
\end{proof}

Proceeding along similar lines as before we show that $K_{n,r}$ is a sub-Poisson random variable, this result is implicitly proved in the Supplementary material by \cite{Fad(17)}, but for the sake of completeness we report it also here.
\begin{proposition} \label{prop:laplace_k}
For any $r \geq 1$ and $n \geq 1$, the random variable $K_{n,r}$ is sub-Poisson with variance factor $\E [K_{n,r}]$. Indeed, for any $\lambda \in \R$ the following bound holds true
\begin{equation} \label{eq:laplace_k}
\log \E [e^{\lambda(K_{n, r}-\E [K_{n, r}] )}]  \leq \phi (\lambda) \E [K_{n,r}],
\end{equation}
where $\phi (\lambda):= e^\lambda-1-\lambda$.
\end{proposition}
\begin{proof}
Exploiting the independence of the random variables $X_{n,j}$'s, for any $\lambda \in \R $ we can write:
\begin{align*}
\log \E [e^{\lambda(K_{n, r}-\E [K_{n, r}] )}] & =
\sum_{j=1}^{\infty} \log \E \exp \left\{ \lambda(\ind{\{ X_{ n,j} = r \}} - \E
\ind{\{ X_{ n,j} =r \}}  ) \right\} \\
& = \sum_{j=1}^{\infty}  \left\{ -\lambda \P (X_{n,j} =r)+\log (e^{\lambda} \P (X_{n,j} =r)+1-\P (X_{n,j} =r))  \right\}\\
& \leq \sum_{j=1}^{\infty}  \phi (\lambda) \P (X_{n,j} = r) = \phi (\lambda) \E [K_{n,r}]
\end{align*}
where we have used the inequality $\log (z) \leq z-1$, for any $z >0$. 
\end{proof}
The previous proposition and the Chernoff bounds imply an exponential tail bound for $K_{n,r}$, indeed one can prove that
\begin{corollary} \label{cor:concentration_K}
For any $n \geq 1$,  $r \geq 1$ and $x \geq 0$ the following holds true 
\begin{equation} \label{eq:concentration_K}
\P (|K_{n,r}-\E [K_{n,r}]| \geq x) \leq 2\exp \left\{ -\frac{x^2}{2(\E [K_{n,r}]+x/3)}
\right\}.
\end{equation}
\end{corollary}
Corollary \ref{cor:concentration_M} and \ref{cor:concentration_K} provide us with concentration inequalities of the missing mass and the statistic $K_{n,r}$, respectively, around their mean. These results have been derived without any assumption on the probabilities $(p_j)_{j\geq 1}$ and hold for all elements of $\Pcr$. In the next Section,  we will focus on the class of regularly varying probabilities and, after recalling the  nonparametric estimator proposed by \citet{Fad(18)} we will prove that this estimator is  consistent within such a subset of $\Pcr$.

\section{A consistent estimator for regularly varying feature probabilities} \label{sec:regular_variation}

\citet{Fad(18)} have introduced a  nonparametric estimator of the missing  mass, defined as follows
\begin{equation}
\label{eq:estimator}
\hat{M}_n := \frac{K_{n,1}}{n}.
\end{equation}
Namely, $\hat{M}_n$ is the number of features having frequency one divided by the sample size $n$. Such an estimator is attractive both from a theoretical and  a computational standpoint. Indeed, on the one side, it admits two different interpretations as a Jackknife estimator in the sense of \cite{Que(56)} and as a non-parametric empirical Bayes estimator in the same spirit as \cite{Efr(73)}; on the other side, it is feasible and easy to implement. See \citet{Fad(18)} for details. 
Here we want to study the consistency of \eqref{eq:estimator}. In order to do this we have seen that, without assumptions on the features' proportions, any  estimator of the missing mass is always inconsistent (Theorem \ref{thm:inconsistency}), hence we study the consistency of \eqref{eq:estimator} under the ubiquitous assumption of heavy tailed probabilities $(p_{j})_{j\geq1}$. We 
rely on the theory of regular variation by \cite{Kar(30),Kar(33)} (see also \cite{Kar(67)}) to
define a suitable class of heavy-tailed $(p_{j})_{j\geq1}$, showing that, under this class, $\hat{M}_n$ turns out to be multiplicative consistent.

We use the limiting notation $f\simeq g$ to mean $f/g\rightarrow1$; we further write $f \lesssim g$ if there exists a fixed constant $C>0$ such that
$f \leq C g$. Then, similarly as done by \cite{Kar(67)} we give the following
\begin{definition} \label{def:rv}
Let $\nu (\ddr x):= \sum_{i\geq1} \delta_{p_i} (\ddr x)$ and define the measure $\overline{\nu} (x):=\nu [x,1]$, which is the cumulative count of all features having no less than a certain probability mass. We say that $(p_j)_{j\geq1}$ is regularly varying with regular variation index $\alpha \in (0,1)$ if $\overline{\nu} (x) \simeq x^{-\alpha}\ell(1/x) $ as $x \downarrow 0$, where $\ell(t)$ is a slowly varying function, that is $\ell (ct)/\ell (t) \to 1$ as $t \to +\infty$ for all $c > 0$.
\end{definition}
Let us remark that if we denote $(p_{[j]})_{j\geq 1}$ the sorted probabilities in decreasing order, definition \ref{def:rv} is equivalent to 
$$ p_{[j]} \simeq j^{-1/\alpha} \ell_*(j),$$
as $j \rightarrow \infty$, where $\ell_*$ is another slowly varying function. For simplicity, the relation between $\ell$, $\ell_*$ and $\alpha$ is skipped here, interested readers can refer to Lemma 22 and Proposition 23 of \cite{Gne(07)}.
Definition \ref{def:rv} is in the same spirit as \cite{Kar(67)}, but for our purposes here we consider the case $\sum_{j \geq 1}p_j< +\infty$, while
in \cite{Kar(67)} the $p_j$'s satisfy the more restrictive condition $\sum_{j\geq 1}p_j=1$. 
The next theorem is similar to a result proved by  \cite{Kar(67)} and provides the  first order asymptotic of $\E K_{n,r}$.
\begin{theorem} \label{thm:karlin}
Let $(p_{j})_{j\geq1}$ be regularly varying with $\alpha\in(0,1)$. If $\Gamma(\cdot)$ denotes the Gamma function, then as $n \to +\infty$, 
$ \E[K_{n, r}] \simeq \frac{\alpha \Gamma (r-\alpha)}{r!}n^\alpha \ell (n)$.
\end{theorem}
\begin{proof}
It is worth to recall the notation already used in Section \ref{sec:consentration}
\[
\Phi_{n,r} := \sum_{j\geq 1} \frac{(np_j)^r}{r!} e^{-np_j}, \quad r \geq 1
\]
roughly speaking $\Phi_{n,r}$ can be considered an asymptotic approximation of $\E K_{n,r}$. Indeed, 
in order to prove the theorem, we first show that $\Phi_{n,r} \simeq \frac{\alpha \Gamma (r-\alpha)}{r!}n^\alpha \ell (n)$
as $n \to +\infty$, and then we prove that $\Phi_{n,r}  \simeq \E [K_{n,r}]$. 
In order to prove the former asymptotic equivalence it is worth noticing that  \cite[Proposition 13]{Gne(07)} applies also for the feature setting  under regularly varying heavy tails, indeed the measure defined by $\nu_r (\D p ):= p^r \nu (\D p)$ is such that 
\begin{equation} \label{eq:nu}
\nu_r ([0,p]) \simeq \frac{\alpha}{r-\alpha} p^{r-\alpha} \ell (1/p), \qquad \text{as } p \to 0.
\end{equation}
Since $\Phi_{n,r} = n^r/r!\int_0^1 e^{-np} \nu_r (\D p)$ is the Laplace transform of $\Phi_{n,r}$ multiplied by a suitable quantity, we can apply Tauberian theorems  to connect the asymptotic behaviour of the cumulative distribution function of $\nu_r$ given in \eqref{eq:nu} to that of $\Phi_{n,r}$. In particular, from Tauberian theorems (see \cite{Fel(71)}), we obtain
\begin{equation} \label{eq:Phi_asym}
\Phi_{n,r} = \frac{n^r}{r!}\int_0^1 e^{-np} \nu_r (\D p) \simeq \frac{n^r}{r!}\alpha \Gamma (r-\alpha) n^{-(r-\alpha)}\ell (n)
= \alpha\frac{\Gamma (r-\alpha)}{r!}  n^\alpha \ell (n),
\end{equation} 
as $n \to +\infty$. As a byproduct of \eqref{eq:Phi_asym}, we get $\Phi_{n,r} \to +\infty$.
Finally to  show $\Phi_{n,r} \simeq \E [K_{n,r}]$, we can easily observe that \cite[Lemma 1]{Gne(07)} applies in this setting as well, hence there exists a constant $c$ such that
\begin{equation}
\label{eq:approx_E}
|\E [K_{n,r}]-\Phi_{n,r}| \leq \frac{c}{n} \max \left\{ \Phi_{n,r} ,\Phi_{n,r+2} \right\} \to 0,
\end{equation}
as $n \to +\infty$. From \eqref{eq:approx_E}, along with  $\Phi_{n,r} \to +\infty$, we obtain
\[
\Big| \frac{\E [K_{n,r}]}{\Phi_{n,r}} -1 \Big| = \frac{|\E [K_{n,r}]-\Phi_{n,r}|}{\Phi_{n,r}} \to 0, \quad
\text{as } n \to +\infty,
\]
in other words we have shown  that $\Phi_{n,r} \simeq \E [K_{n,r}]$ as $n \to +\infty$.

\end{proof}
We are now ready to prove that $\hat{M}_n$ is multiplicative consistent, when the feature probabilities $(p_{j})_{j\geq1}$
are regularly varying. In the proof  we will use the concentration inequalities of Section \ref{sec:consentration} along with Theorem \ref{thm:karlin} to tune the concentration inequalities under the assumption of regular variation.
\begin{proposition}
Let $(p_{j})_{j\geq1}$ be regularly varying with index $\alpha\in(0,1)$.
Let $\hat{M}_n:= K_{n,1}/n$ be the nonparametric  estimator of the missing mass in a sample of size $n$, then $\hat{M}_n$ is strongly multiplicative consistent, 
i.e. $M_n/\hat{M}_n \stackrel{a.s.}{\longrightarrow}1$. 
\end{proposition}
\begin{proof}
In order to prove the multiplicative consistency we first  show that $K_{n,1}/\E[K_{n,1}]\stackrel{a.s.}{\longrightarrow}1$
and that $M_n/\E[M_n]\stackrel{a.s.}{\longrightarrow} 1$. As for the former convergence, we can use the concentration inequality \eqref{eq:concentration_K} given in Corollary \ref{cor:concentration_K} when $r=1$, which, for any $\varepsilon>0$, gives
\begin{equation} \label{eq:conc_K1}
\P (|K_{n,1}/\E [K_{n,1}]-1| \geq \varepsilon) \leq 2\exp \left\{ -\frac{\varepsilon^2\E [K_{n,1}]}{2(1+\varepsilon \E [K_{n,1}]/3)}\right\}.
\end{equation}
When $\varepsilon>0$ is fixed, we can use the asymptotic $ \E[K_{n, 1}] \simeq \alpha \Gamma (1-\alpha)n^\alpha \ell (n)$ in  Theorem \ref{thm:karlin} to say that
\begin{equation*}
\sum_{n \geq 1} \P (|K_{n,1}/\E [K_{n,1}]-1| \geq \varepsilon) \stackrel{\eqref{eq:conc_K1}}{\lesssim}
\sum_{n \geq 1} 2\exp \left\{- n^\alpha \ell (n)  \right\} < +\infty,
\end{equation*}
which implies that for any $\varepsilon>0$, $\P (\limsup_{n} ( |K_{n,1}/\E [K_{n,1}]-1| \geq \varepsilon ))=0$ by the first Borel-Cantelli lemma, hence $K_{n,1}/\E[K_{n,1}]\stackrel{a.s.}{\longrightarrow}1$.\\
Analogously we may use Corollary \ref{cor:concentration_M} to prove the almost sure convergence to $1$ of the ratio
$M_n/\E [M_n]$. Indeed, for any $\varepsilon>0$, we have
\begin{align*}
&\P (|M_n/\E [M_n]-1| \geq \varepsilon)\\
 & \qquad\leq \P (M_n -\E [M_n]\geq \varepsilon \E [M_n]) +\P (M_n-\E [M_n]\leq 
-\varepsilon \E M_n)\\
 & \qquad\leq \exp\left\{ -v_n^+n^2 \left[ 1+\frac{\varepsilon \E [M_n]}{n v_n^+} -\sqrt{1+\frac{\E [M_n]}{n v_n^+}}  \right] \right\}  +\exp \left\{ -\frac{\varepsilon^2 (\E [M_n])^2}{2 v_n^-} \right\}.
\end{align*}
By observing that $\E [M_n]= \E [K_{n+1,1}]/(n+1)$, the previous upper bound boils down to
\begin{equation}
\begin{split}
\P (|M_n/\E [M_n]-1| \geq \varepsilon) &\leq \exp\left\{ -v_n^+n^2 \left[ 1+\frac{\varepsilon \E [K_{n+1,1}]}{n (n+1) v_n^+} -\sqrt{1+\frac{\E [K_{n+1,1}]}{n (n+1)v_n^+}}  \right] \right\}\\
& \qquad\qquad\qquad\qquad\qquad\quad+\exp \left\{ -\frac{\varepsilon^2 (\E[ K_{n+1,1}])^2}{2 (n+1)^2 v_n^-} \right\}.
 \end{split}
\end{equation}
Now, using again Theorem \ref{thm:karlin}, it is not difficult to see that for any fixed $\varepsilon>0$
\begin{equation}
\sum_{n \geq 1} \P (|M_n/\E [M_n]-1| \geq \varepsilon) \lesssim \sum_{n \geq 1}\exp \left\{- n^\alpha \ell (n)  \right\} <+\infty 
\end{equation}
then, by the first Borel-Cantelli lemma, we get $M_n/\E [M_n]\stackrel{a.s.}{\longrightarrow}1$, as well.\\
By the previous results the consistency of $\hat{M}_n$ easily follows, indeed 
\begin{equation*}
\frac{M_n}{\hat{M}_n} = \frac{M_n}{\E [M_n]} \cdot \frac{n\E [M_n]}{\E [K_{n,1}]} \cdot \frac{ \E [K_{n,1}]}{K_{n,1}} 
\stackrel{a.s.}{\longrightarrow}1,
\end{equation*}
since all the ratios on the r.h.s. converge to $1$ almost surely.
\end{proof}

\end{document}